\documentclass[a4paper,12pt]{amsart}

\usepackage{amssymb}
\usepackage{amsmath}
\usepackage{amsthm}
\usepackage{mathtools}
\usepackage{fullpage}
\usepackage{tikz-cd}
\usepackage{tikz}
\usetikzlibrary{calc}
\usepackage[l2tabu,orthodox]{nag}
\usepackage[all, warning]{onlyamsmath}
\usepackage{here}
\usepackage{stmaryrd}
\usepackage{graphicx}
\usepackage{enumitem}
\usepackage{xcolor} 
\usepackage{longtable}
\usepackage{adjustbox}
\usepackage{bm}
\usepackage{mathdots}
\usepackage{setspace}
\usepackage[breaklinks=true,colorlinks=true,linkcolor=blue,citecolor=teal]{hyperref}
\usepackage[all]{xy}

\usetikzlibrary{decorations.markings}
\usetikzlibrary{shapes.geometric}
\tikzset{every loop/.style={min distance=10mm,looseness=10}}
\usetikzlibrary{decorations.markings}
\tikzcdset{scale cd/.style={every label/.append style={scale=#1},
    cells={nodes={scale=#1}}}}

\numberwithin{equation}{section}
\numberwithin{figure}{section}
\numberwithin{table}{section}
\allowdisplaybreaks

\setenumerate{label=(\arabic*),nosep}
\setitemize{nosep}
\newlist{clist}{enumerate}{1}
\setlist*[clist]{label=(\roman*), nosep}

\theoremstyle{plain}
\newtheorem{thm}{Theorem}[section]
\newtheorem{lem}[thm]{Lemma}
\newtheorem{prp}[thm]{Proposition}
\newtheorem{cor}[thm]{Corollary}

\theoremstyle{definition}

\newtheorem{dfn}[thm]{Definition}

\newtheorem{eg}[thm]{Example}
\newtheorem{rmk}[thm]{Remark}

\theoremstyle{definition}
\newtheorem*{ackn}{Acknowledgements}
\newtheorem*{org}{Organization}

\usepackage{cleveref}
\usepackage{autonum}
\crefname{thm}{Theorem}{Theorems}
\crefname{asm}{Assumption}{Assumptions}
\crefname{cor}{Corollary}{Corollaries}
\crefname{dfn}{Definition}{Definitions}
\crefname{fct}{Fact}{Facts}
\crefname{lem}{Lemma}{Lemmas}
\crefname{mth}{Theorem}{Theorem}
\crefname{ntn}{Notation}{Notations}
\crefname{prp}{Proposition}{Propositions}
\crefname{rmk}{Remark}{Remarks}
\crefname{eg}{Example}{Examples}
\crefname{section}{\S\!}{\S\S\!}
\crefname{equation}{equation}{equations}
\crefname{exc}{Exercise}{Exercises}

\newcommand{\tit}{\textit}
\newcommand{\ol}{\overline}
\newcommand{\ul}{\underline}

\newcommand{\spn}[1]{\langle #1 \rangle}

\newcommand{\bbC}{\mathbb{C}}
\newcommand{\bbF}{\mathbb{F}}
\newcommand{\bbM}{\mathbb{M}}
\newcommand{\bbN}{\mathbb{N}}
\newcommand{\bbQ}{\mathbb{Q}}
\newcommand{\bbZ}{\mathbb{Z}}

\newcommand{\bfa}{\mathbf{a}}
\newcommand{\bfb}{\mathbf{b}}
\newcommand{\bfp}{\mathbf{p}}
\newcommand{\bfs}{\mathbf{s}}

\usepackage{mathrsfs}
\newcommand{\scC}{\mathscr{C}}
\newcommand{\scH}{\mathscr{H}}
\newcommand{\scL}{\mathscr{L}}

\DeclareMathOperator{\Aut}{Aut}
\DeclareMathOperator{\Ker}{Ker}
\DeclareMathOperator{\Gal}{Gal}
\DeclareMathOperator{\Ind}{Ind}
\DeclareMathOperator{\Irr}{Irr}

\makeatletter
\def\th@remark{%
  \thm@headfont{\bfseries}%
  \normalfont 
  \thm@preskip\topsep \divide\thm@preskip\tw@
  \thm@postskip\thm@preskip
}
\makeatother

\begin{document}

\title[Spanning trees and their relations in Galois covers]
{Spanning trees and their relations in Galois covers}

\author{Kosuke Mizuno} 

\address{GRADUATE SCHOOL OF MATHEMATICS, NAGOYA UNIVERSITY, FURO-CHO, CHIKUSAKU, NAGOYA, 464-8601, JAPAN}
\email{kosuke.mizuno.c1@math.nagoya-u.ac.jp}

\begin{abstract}
This paper studies the relation among the number of spanning trees of intermediate graphs in a Galois cover, building on results for $(\bbZ/2\bbZ)^m$-covers previously established by Hammer, Mattman, Sands, and Valli\`{e}res. We generalize their results to arbitrary finite Galois covers. Using the Ihara zeta function and the Artin--Ihara $L$-function, we prove two formulas which are graph-theoretic analogues of Kuroda's formula and the Brauer--Kuroda relations in algebraic number theory. Furthermore, we prove that a spanning tree formula does not exist if the Galois group is cyclic.   
\end{abstract} 
\maketitle 
\tableofcontents 

\section{Introduction}

 Let $X$ be a finite connected graph. The number of spanning trees of $X$ is denoted by $\kappa(X)$ and called the complexity of $X$. Let $Y/X$ be a Galois covering with Galois group $G=\Gal(Y/X)$ (see Definition \ref{dfn:Gal}). We are interested in the relation among the complexities of intermediate graphs of $Y/X$. In \cite{HKSV}, Hammer, Mattman, Sands, and Valli\`{e}res prove the spanning tree formula in which the complexity of $Y$ can be expressed by those of the intermediate graphs of $Y/X$ in the case $\Gal(Y/X)=(\bbZ/2\bbZ)^{m}$. More precisely, they give the following formula (see \cref{thm:HMSV}):
$$\kappa(Y) =\frac{2^{2^{m}-m-1}}{\kappa(X)^{2^{m}-2}}\prod_{i=1}^{2^{m}-1}\kappa(X_i),$$
where $X_{i}$ are the intermediate graphs with $\Gal(X_i/X)=\bbZ/2\bbZ$ for $i =1,\ldots, 2^{m}-1$. In the case $m=2$, this result is an analogue of Kuroda's formula in algebraic number theory (cf.~\cite{kurodarel1, kurodarel2, kurodarrel3}). 

In this paper, we generalize the formula in \cite{HKSV} to arbitrary finite Galois covers. We prove two spanning tree formulas that impose no conditions on Galois groups. Our two formulas lead to the result of \cite{HKSV} in the case $\Gal(Y/X)=(\bbZ/2\bbZ)^m$ (see Examples \ref{kuroda eg} and \ref{bru val}).

First, we state the first theorem. Let $G$ be a finite group, and let $$\text{$\scH_{G}\coloneq \{H \mid H=\Ker\rho$ for some $\rho\in \Irr(G) \}$},$$ where $\Irr(G)$ is the set of irreducible representations of $G$. By adding the empty set $\emptyset$ to this set, we define $$\ul{\scH_{G}}\coloneq \scH\cup\{\emptyset\}. $$ This set becomes a partially ordered set by the inclusion relation. Our first theorem is a spanning tree formula using the M\"{o}bius function of $\ul{\scH_{G}}$. 


\begin{thm}[$=$ \cref{main1}]\label{main1n}
	Let $G$ be a finite group, and let $\scH_{G}$ be the partially ordered set defined above. Let $\mu\colon \ul{\scH_{G}}\times\ul{\scH_{G}}\to \bbZ$ be the M\"{o}bius function of $\ul{\scH_{G}}$ (cf.~Appendix \ref{mobi fu}). Then, for any Galois cover $Y/X$ with $\Gal(Y/X)=G$, we have   
\[	
\kappa(Y)=\frac{1}{|G|}\prod_{H\in \scH_{G}}([G:H]\cdot\kappa(X_H))^{-\mu(\emptyset ,H)},
\]
where $X_H$ denotes the intermediate graph corresponding to $H$.
\end{thm}


Second, we state a second theorem. This theorem gives an analogue of the Brauer--Kuroda relations in algebraic number theory, which is a product formula for Dedekind zeta functions (cf.~\cite{Brarelkuroda, kurodarel1}). 
We also let a partially ordered set as follows. The set $\scC_{G}$ is given by
$$\scC_{G} \coloneq \{ C  \mid C \text{ is a cyclic subgroup of $G$} \}.$$ The set $\scC_{G}$ becomes a partially ordered set by inclusion relation. We define $$\ol{\scC_{G}}\coloneq \scC_{G} \cup \{\infty\}, $$where $\infty$ is an abstract element greater than any $C\in \scC_{G}$. Our second theorem is a spanning tree formula using the M\"{o}bius function of $\ol{\scC_{G}}$.


\begin{thm}[$=$ \cref{main2}]\label{main2n}
Let $G$ be a finite group, and let $\ol{\scC_{G}}$ be the partially ordered set defined above. Let $\mu\colon \ol{\scC_{G}}\times\ol{\scC_{G}}\to \bbZ$ be the M\"{o}bius function of $\ol{\scC_{G}}$ (cf.~Appendix \ref{mobi fu}). Then, for any Galois cover $Y/X$ with $\Gal(Y/X)=G$, we have   
\[
\kappa(X)=\prod_{ C\in \scC_{G}}([G:C]\cdot \kappa(X_C))^{-\frac{1}{[G:C]} \mu(C,\infty)},
\] 
where $X_C$ denotes the intermediate graph corresponding to $C$.
\end{thm}

In \cref{main2n}, if $\mu(\{1_G\},\infty)=0$, then $\kappa(Y)$ does not appear on the right-hand side of the equation. In this case, \cref{main2n} does not give a formula expressing $\kappa(Y)$. In \cite{excep}, a group $G$ with the condition $\mu(\{1_{G}\},\infty)=0$ is called \tit{exceptional} (see Remark \ref{rmk:excrem} (2), Example \ref{eg:q8}). 


Theorems \ref{main1n} and \ref{main2n} lead to the formula in \cite{HKSV} in the case $\Gal(Y/X)=(\bbZ/2\bbZ)^{m}$ (see Examples \ref{kuroda eg} and \ref{bru val}). Furthermore, we succeed in removing the assumption that the Euler characteristic is non-zero.  Therefore, Theorems \ref{main1n} and \ref{main2n} can be seen as a generalization of the result in \cite{HKSV}. 



Finally, we state the third theorem. If a group $G$ has a faithful irreducible representation, then $G$ is called \tit{irreducibly represented}. $G$ is irreducibly represented if and only if \cref{main1n} implies the trivial equation $\kappa(Y)=\kappa(Y)$ (see Remark \ref{rmk:faithful}). $G$ is a cyclic group if and only if \cref{main2n} implies the trivial equation $\kappa(X)=\kappa(X)$ (see Remark \ref{rmk:excrem} (1)). On the other hand, if $G$ is abelian but not cyclic, Theorems \ref{main1n} and \ref{main2n} give non-trivial formulas. Hence, in the case where $G$ is cyclic, the natural question arises whether it is possible to exist a monomial-type formula. \cref{main3n} is an answer to this question. \cref{main3n} states that there does not exist a monomial-type formula if $G$ is a non-trivial cyclic group.

\begin{thm}[$=$ \cref{main3}]\label{main3n} Let $G$ be a nontrivial cyclic group, and let $H_{1},\ldots,H_{k}$ be all of its subgroups. Then there do not exist $\bm{0}\neq (m_1,\ldots,m_{k})\in \bbZ^{k}$ and $q\in \bbQ$ satisfying the following condition: For any Galois cover $Y/X$ with $\Gal(Y/X)=G$, it follows that $$q\cdot \prod_{i=1}^{k}\kappa(X_{i})^{m_{i}}=1,$$ where $X_{i}$ is the intermediate graph corresponding to $H_{i}$. 
\end{thm}

In summary, we can see from the main results whether a non-trivial monomial-type formula exists for any finite group. On the other hand, it is difficult to see whether there exists a formula expressing $\kappa(Y)$ for a finite group $G$, but we answer this question in the case where $G$ is abelian by main theorems.   

\begin{org}
The paper is organized as follows. In \cref{graph}, we introduce basic concepts related to graphs, including the definition of graphs in \cref{gr pre} and Galois theory for graphs in \cref{gal pre}. In \cref{zeta}, we present the Ihara zeta function and the Artin--Ihara $L$-function, which are key tools in our proofs. In \cref{span formula}, we prove the main theorems. We prove \cref{main1} in \cref{kro}, \cref{main2} in \cref{bkru}, and \cref{main3} in \ref{exist}. In \cref{mobi fu}, we introduce the M\"{o}bius function and M\"{o}bius inversion, which are used in the main results. 
\end{org}

\begin{ackn}
	The author would like to thank his supervisor, Prof.~Yusuke Nakamura, for his thoughtful guidance and insightful suggestions. He is also grateful to Ryosuke Murooka for his assistance and valuable comments. Furthermore, I would like to thank Taiga Adachi, Prof.~Futaba Fujie, Prof.~Hiroshi Suzuki, and Sohei Tateno for their valuable comments and suggestions.    
\end{ackn}

\section{Graph Theory}\label{graph}

In this section, we introduce the notion of graphs.
\subsection{Definition of graphs}\label{gr pre}

In this paper, we use the definition of graphs given by Serre and refer the reader to \cite{Sunada} for further details.

\begin{dfn}[\cite{Ser77}, \cite{Sunada}]
	A \tit{graph} $X=(V_X, E_X,o,t,\iota)$ consists of two sets $V_X$, $E_X$ and three maps $o\colon E_X\to V_X, t\colon E_X\to V_X, \iota\colon E_X\rightarrow E_X $ satisfying \begin{enumerate}
	\item $\iota\circ\iota(e)=e$ for all $e\in E_X$,
	\item $\iota(e)\neq e$ for all $e\in E_X$,
	\item $o(\iota(e))=t(e)$ for all $e\in E_X$,
	\item $t(\iota(e))=o(e)$ for all $e\in E_X$.
\end{enumerate}
\end{dfn} 
An element of $V_X$ is called a \tit{vertex}, and an element of $E_X$ is called an \tit{edge}. The maps $o$ and $t$ are called the \tit{origin map} and \tit{terminus map}, respectively, while the map $\iota$ is called the \tit{inversion map}. We often write $\ol{e}$ instead of $\iota(e)$ for $e \in E_X$. For $v\in V_X$, we define $$E_{X,v}\coloneq \{e\in E_X\mid o(e)=v\}.$$   The \tit{Euler characteristic} $\chi(X)$ of $X$ is defined as $$\chi(X)\coloneq |V_X|-\frac{|E_X|}{2}. $$ We can choose a subset $S$ of $E_X$ such that $E_X=S\cup \ol{S}$ and $S\cap \ol{S}=\emptyset$. This subset $S$ is called an \tit{orientation} of $X$. 
\begin{dfn}
	Let $Y=(V_Y, E_Y)$ and $X=(V_X, E_X)$ be two graphs. A \tit{morphism} of graphs $f=(f_V,f_E)\colon Y\rightarrow X$ from $Y$ to $X$ is a pair $f=(f_{V},f_{E})$ of maps $f_V\colon  V_Y\rightarrow V_X$ and $f_E\colon E_Y\rightarrow E_X$ satisfying the following conditions. \begin{enumerate}
	\item $f_V(o(e))=o(f_E(e))$ for all $e\in E_Y$,
	\item $f_V(t(e))=t(f_E(e))$ for all $e\in E_Y$,
	\item $\ol{f_E(e)}=f_E(\ol{e})$ for all $e\in E_Y$.
\end{enumerate}
\end{dfn}
A morphism $f$ is called an \tit{isomorphism} if $f_V$ and $f_E$ are bijective. In this case, $Y$ and $X$ are said to be \tit{isomorphic}. For simplicity of notation, we write $f$ instead of $f_V$ and $f_E$. 

\subsection{Galois theory for graphs}\label{gal pre} 
In this section, we introduce the notions of Galois covers and the fundamental theorem of Galois theory for graphs. For more details, we refers the reader to \cite{Terras} and \cite{Sunada}. We then present the concept of voltage assignments, which is used to construct a Galois cover with an arbitrary group $G$ as the Galois group. For further details, see \cite{GT01} and \cite{ray}.  
\begin{dfn}
	Let $Y$ and $X$ be graphs. A morphism of graphs $\pi\colon Y\rightarrow X$ is called a \textit{covering map} (or simply a \textit{cover}) if the following two conditions are satisfied:\begin{enumerate}
	\item $\pi\colon V_Y\rightarrow V_X$ is surjective,
	\item For each $w\in V_Y$, the restriction of $\pi$ to $E_{Y,w}$ induces a bijective map $\pi|_{E_{Y,w}}\colon E_{Y,w}\rightarrow E_{X, \pi(w)}.$
\end{enumerate}
\end{dfn}  
\noindent The group $\Aut(Y/X)$ is defined by $$\Aut(Y/X)\coloneq \{\sigma\colon Y \to Y\mid \text{$\sigma$ is an isomorphism such that } \pi\circ\sigma=\pi \}.$$ If $\pi: Y\to X$ is a cover, we denote it briefly by $Y/X$.

\begin{dfn}\label{dfn:Gal}
	A covering map $\pi\colon Y\rightarrow X$ is called a \tit{Galois cover} if the following two conditions are satisfied. \begin{enumerate}
		\item $Y$ is connected.
		\item $\Aut(Y/X)$ acts transitively on each fiber $\pi^{-1}(v)$ for $v\in V_X$. 
	\end{enumerate}
\end{dfn}
In this case, we write $\Gal(Y/X)$ instead of $\Aut(Y/X)$. A graph $Z$ is called an \textit{intermediate covering} of $\pi$ if there exist two covering maps $\pi_2: Y\rightarrow Z$ and $\pi_1: Z\rightarrow X$  with $\pi=\pi_1\circ \pi_2$. There is a one-to-one correspondence between the subgroups of $\Gal(Y/X)$ and the equivalence classes of intermediate graphs of $Y/X$ (see Theorem 14.3 in  \cite{Terras}). We denote by $X_{H}$ the intermediate graph corresponding to the subgroup $H$ of the Galois group $\Gal(Y/X)$.     
Let $X_{H}$ and $X_{H'}$ be intermediate graphs of $\pi$. Then $H$ and $H'$ are conjugate if and only if $X_{H}$ and $X_{H'}$ are isomorphic as covers (see Theorem 14.5 in \cite{Terras}).  

We define the voltage assignments, which allow any finite group to be realized as a Galois group.
\begin{dfn}
 Let $X=(V_X,E_X,o,t,\iota)$ be a finite graph, and fix an orientation $S$. Let $G$ be a finite group and $\alpha\colon S\to G $ is a map. Then we extend $\alpha$ to $\ol{S}$ by setting $\alpha(\ol{s})=\alpha(s)^{-1}$ for $s\in S $.
Then a graph $X(\alpha)=(V_{X(\alpha)}, E_{X(\alpha)}, o, t, \iota)$ is defined as follows. 
\[
V_{X(\alpha)}=V_X\times G,\quad E_{X(\alpha)}=E_X\times G,
\]  
\[
o((e,\sigma)) =(o(e),\sigma), \quad t((e,\sigma))=(t(e),\sigma\cdot\alpha(e)),\quad \iota((e,\sigma))=(\iota(e),\sigma\cdot \alpha(e)).
\]
The graph $X(\alpha)$ is called the \tit{derived graph} of $X$, and $\alpha$ is called a \tit{voltage assignment}.
\end{dfn}

We define a morphism $\pi=(\pi_V,\pi_E)\colon X(\alpha)\to X$ by 
\[
\pi_V\colon V_{X(\alpha)}\ni (v,\sigma)\mapsto v\in V_X,\qquad \pi_{E}\colon E_{X(\alpha)}\ni (e,\sigma)\mapsto e\in E_X .
\]

\noindent Then $\pi$ is a covering map. Moreover, $\pi$ is a Galois cover with $\Gal(X(\alpha)/X)=G$ if $X(\alpha)$ is connected. Consequently, any finite group can be realized as a Galois group.


\section{Zeta Functions for Graphs}\label{zeta}
In this section, we introduce the Ihara zeta functions and the Artin--Ihara $L$-functions. For further details, we refer the reader to \cite{Terras}.
\subsection{Ihara zeta functions}

The Ihara zeta function is a rational function associated with a finite connected graph $X$, denoted by $\zeta_{X}(u)$. One notable property of the Ihara zeta function is that it can be explicitly computed using the three-term determinant formula, which expresses $\zeta_{X}(u)$ in terms of the adjacency matrix $A_X$ and degree matrix $D_X$ (see Theorem 2.5. \cite{Terras}): 
$$\zeta_{X}(u)=\frac{(1-u^{2})^{\chi(X)}}{\det(I-A_{X}u+(D_{X}-I)u^{2}).}$$  

Hashimoto establishes a connection between the Ihara zeta function and the number of spanning trees of the graph. By defining $h_X(u)\coloneq \det(I - A_{X}u + (D_{X}-I)u^2)$, we obtain the following result (see \cite{Hashimoto}):
$$h'_X(1)=-2\chi(X)\kappa(X), $$ where $\kappa(X)$ denotes the number of spanning trees of $X$. 
\subsection{Artin--Ihara $L$-functions}
The Artin--Ihara $L$-function is a rational function associated with a Galois cover $Y/X$ and a representation of its Galois group $G=\Gal(Y/X)$. Similar to the Ihara zeta function, the Artin--Ihara $L$-function can be expressed by two matrices $A_{\rho}$, $D_{\rho}$, which are associated with a representation $\rho$ and $Y/X$ (see Theorem 18.15 in \cite{Terras}):

$$L_{Y/X}(u, \rho)=\frac{(1-u^2)^{\chi(X)\deg\rho}}{\det(I-A_{\rho}u+(D_{\rho}-I)u^2)}.$$

 Similar to the Ihara zeta function, we define $$h_{Y/X}(u,\rho)\coloneq \det(I-A_{\rho}u+(D_{\rho}-I)u^2).$$      


In this paper, the trivial representation is denoted by $\rho_{0}$. Here are some of the main properties of the Artin--Ihara $L$-function.

\begin{prp}[cf.~{\cite[Proposition 18.10]{Terras}}]\label{artinL pro}
	Let $Y/X$ be a Galois covering with Galois group $G=\Gal(Y/X)$. Then we have the following properties. 
\begin{enumerate}
	\item If $\rho$ and $\tau$ are representations of $G$, then $$L_{Y/X}(u,\rho\oplus\tau)=L_{Y/X}(u,\rho)L_{Y/X}(u,\tau). $$
	\item Let $H$ be a normal subgroup of $G$, and let $\tilde{\rho}$ be a representation of $G/H=\Gal(X_{H}/X)$. Then $$L_{Y/X}(u,\rho)=L_{X_{H}/X}(u,\tilde{\rho}), $$ where $\rho$ is the lift of $\tilde{\rho}$. 
	\item Let $\rho$ be a representation of $H$. Then $$L_{Y/X}(u,\Ind^{G}_{H}(\rho))=L_{Y/X_{H}}(u,\rho),$$ where $\Ind^{G}_{H}(\rho)$ is the induced representation of $\rho$.  
\end{enumerate}	 
\end{prp}


In addition to these properties, the Artin--Ihara $L$-function also relates the number of spanning trees of the graph $Y$ to those of $X$:

\begin{prp}[cf.~{\cite[\S 3.2.4]{pengo}}]\label{formula}
	Let $Y/X$ be a Galois cover with Galois group $G=\Gal(Y/X)$, and assume that $\chi(X)\neq 0$. Then $$|G|\cdot\kappa(Y)=\kappa(X)\cdot \prod_{\rho \in \Irr(G)\setminus\{\rho_{0}\} }h_{Y/X}(1,\rho)^{\deg\rho}. $$
\end{prp} 

\section{Spanning tree formulas for Galois covers}\label{span formula}

In \cite{HKSV}, Hammer, Mattman, Sands, and Valli\`{e}res prove the spanning tree formula expressing the complexity $\kappa(Y)$ of $Y$ in terms of $\kappa(X_{i})$'s of the intermediate graphs $X_{i}$ of $Y/X$ when $\Gal(Y/X)=(\bbZ/2\bbZ)^{m}$.

\begin{thm}[\cite{HKSV}, Theorem 3.6, Remark 3.7]\label{thm:HMSV}
	Let $G=(\bbZ/2\bbZ)^{m}$, and let $H_{1},\ldots,H_{2^{m}-1}$ be subgroups of $G$ with index $2$. Then, for any Galois cover $Y/X$ with $\Gal(Y/X)=(\bbZ/2\bbZ)^m$ satisfying $\chi(X)\neq 0$, we have   
\[
\kappa(Y) =\frac{2^{2^{m}-m-1}}{\kappa(X)^{2^{m}-2}}\prod_{i=1}^{2^{m}-1}\kappa(X_i),
\]
where $X_1,\ldots,X_{2^{m}-1}$ are the intermediate graphs corresponding to $H_{1},\ldots,H_{2^{m}-1}$.
\end{thm}
We will generalize the above formula in \cite{HKSV} to arbitrary finite Galois covers.

\subsection{The generalization of Kuroda's formula}\label{kro}
In this section, we prove \cref{main1} and examine several examples.

\begin{dfn}\label{dfn:hg}
For a finite group $G$, we define 
\[
\text{$\scH_{G}\coloneq \{H \mid H=\Ker \rho$ for some $\rho\in \Irr(G) \}$},
\] 
where $\Irr(G)$ denotes the set of isomorphism classes of irreducible representations of $G$. By adding the empty set \(\emptyset\) to this set, we also define 
\[
\ul{\scH_{G}}\coloneq  \scH_{G}\cup\{\emptyset\}.
\]
The set $\ul{\scH_{G}}$ becomes a partially ordered set by the inclusion relation. 	
\end{dfn}


\begin{thm}\label{main1}
	Let $G$ be a finite group, and let $\scH_{G}$ be the partially ordered set defined in Definition \ref{dfn:hg}. Let $\mu\colon \ul{\scH_{G}}\times\ul{\scH_{G}}\to \bbZ$ be the M\"{o}bius function of $\ul{\scH_{G}}$ (cf.~Appendix \ref{mobi fu}). Then, for any Galois cover $Y/X$ with $\Gal(Y/X)=G$, we have   
\[	
\kappa(Y)=\frac{1}{|G|}\prod_{H\in \scH_{G}}([G:H]\cdot\kappa(X_H))^{-\mu(\emptyset ,H)},
\]
where $X_H$ denotes the intermediate graph corresponding to $H$.
\end{thm}


\begin{proof}

If $\chi(X)= 0$, then $\Gal(Y/X)$ is cyclic (see Remark 2.7 in \cite{ray}). Then $\scH_{G} $ contains the trivial group $\{1_{G}\}$ of $G$, and hence, we have $\mu(\emptyset, H) = 0$ for any $H \in \scH_{G}\setminus\{1_G\}$. In this case, the theorem is true. From now on, we assume $\chi(X)\neq 0$.

Define two functions $f,g\colon \ul{\scH_{G}}\to \bbC^{*}
$ as follows


\begin{align}
  f(H) &= 
  \begin{cases}     
    \kappa(X) & \text{if $H=G$},\\
    \displaystyle\prod_{\substack{\rho\in \Irr(G)\setminus\{\rho_0\} \\ H=\Ker(\rho)} } h_{Y/X}(1,\rho)^{\deg\rho} & \text{otherwise},
  \end{cases} \\
  g(H) &= \kappa(X) \cdot \prod_{\substack{\rho\in \Irr(G)\setminus\{\rho_0\} \\ H\subseteq \Ker\rho}} h_{Y/X}(1,\rho)^{\deg\rho}.
\end{align}

\noindent Then we have
\begin{equation}
    g(H)=\prod_{\substack{N\in \scH_{G}; \\H\subseteq N}} f(N).
\end{equation}
for any $H\in \ul{\scH_{G}}$.
By the M\"{o}bius inversion formula (Lemma \ref{Mö inv}), we have that 
\begin{equation}\label{1}
	 f(H)=\prod_{\substack{N\in \scH_{G}; \\H\subseteq N}}g(N)^{\mu(H,N)}. \tag{1}
\end{equation}
By Proposition \ref{formula} and \ref{artinL pro} (2), for any $H\in \scH_{G}$, we have  

\begin{equation}
g(H) \overset{\text{\ref{artinL pro}(2)}}{=}\kappa(X)\cdot \prod_{\tilde{\rho}\in \Irr(G/H)\setminus\{\rho_0\}} h_{X_{H}/X}(1,\tilde{\rho})^{\deg\tilde{\rho}}\overset{\text{\ref{formula}}}{=}[G:H]\cdot \kappa(X_H).\tag{2}\label{2} \end{equation}

\noindent From the above results, we obtain that \begin{align}
	|G|\cdot\kappa(Y) &\overset{\text{\ref{formula}}}{=} \kappa(X)\cdot\prod_{\rho\in \Irr(G)\setminus\{\rho_0\} }h_{Y/X}(1,\rho)^{\deg\rho}\\ &=\prod_{H\in \scH_{G}}f(H)\\ &\overset{\eqref{1}}{=}\prod_{H \in \scH_{G}}\prod_{\substack{N\in \scH_{G}; \\ H\subseteq N}}g(N)^{\mu(H,N)}\\ &\overset{\eqref{2}}{=} \prod_{H\in \scH_{G}}\prod_{\substack{N\in \scH_{G} ;\\ H\subseteq N}}([G:N]\cdot \kappa(X_{N}))^{\mu(H,N)}\\&= \prod_{N\in \scH_{G}}\prod_{\substack{H\in \scH_{G} ;\\ H\subseteq N}}([G:N]\cdot \kappa(X_{N}))^{\mu(H,N)} \\ &\overset{\text{\ref{mob def}}}{=}\prod_{N\in \scH_{G}}([G:N]\cdot\kappa(X_N))^{-\mu(\emptyset,N)}. 
\end{align}
Therefore, we conclude that 
\begin{equation}
	\kappa(Y)=\frac{1}{|G|}\prod_{H\in \scH_{G}}([G:H]\cdot\kappa(X_H))^{-\mu(\emptyset ,H)}.
\end{equation}
\end{proof}


\begin{rmk}\  \label{rmk:faithful}

\begin{enumerate}
	\item If a group $G$ has a faithful irreducible representation, then $G$ is called \tit{irreducibly represented}. For example, cyclic groups are irreducibly represented. In this case, the assertion of \cref{main1} becomes trivial. Indeed, $\scH_{G} $ contains the trivial group $\{1_{G}\}$ of $G$ in this case, and hence, we have $\mu(\emptyset, H) = 0$ for any $H \in \scH_{G}\setminus\{1_G\}$. In Table \ref{figure}, we provide a list of groups with order at most $24$, and we indicate whether they are irreducibly represented or not. We computed Table \ref{figure} using data from \cite{GroupNames}.
	\item When $\chi(X)=0$, it is known that $\Gal(Y/X)=G$ is always a cyclic group, and we have $\kappa(Y)=|G|\cdot\kappa(X)$ (see Remark 2.7 in \cite{ray}).
\end{enumerate}
\end{rmk}

\begin{eg}\label{kuroda eg}
	Let $Y/X$ be a Galois cover with Galois group $G=(\bbZ/2\bbZ)^{m}$. $G$ has $2^{m}-1$ subgroups $H_{1},\ldots,H_{2^{m}-1}$ of index $2$, and we have $$\ul{\scH_{G}}=\{\emptyset,H_1,\ldots,H_{2^{m}-1}, G\}.$$ Furthermore, we have $\mu(\emptyset,G)=2^{m}-2$, and $\mu(\emptyset,H_{i})=-1$ for $1,\ldots,2^{m}-1$. Therefore, by \cref{main1}, we obtain  
	\begin{align}
		\kappa(Y) &=\frac{1}{|G|}([G:G]\cdot\kappa(X)^{-(2^m-2)})\prod_{i=1}^{2^m-1}([G:H_i]\cdot\kappa(X_i))\\
		&= \frac{2^{2^{m}-m-1}}{\kappa(X)^{2^{m}-2}}\prod_{i=1}^{2^{m}-1}\kappa(X_i),
	\end{align}
where $X_1,\ldots,X_{2^{m}-1}$ are the intermediate graphs corresponding to $H_{1},\ldots,H_{2^m-1}$. This formula coincides with the formula obtained by Hammer, Mattman, Sands, and Valli\`{e}res in \cref{thm:HMSV}. Therefore, \cref{main1} can be seen as a generalization of \cref{thm:HMSV}.     
\end{eg}


\begin{eg}

Let $G=\bbZ/2\bbZ\times\bbZ/6\bbZ$. Then we have 
\[
\ul{\scH_{G}}=\{\emptyset, H_{1},\ldots,H_7,G \},
\]
where $H_1=\spn{(1,0)}$, $H_{2}=\spn{(1,3)}$, $H_{3}=\spn{(0,3)}$, $H_{4}=\spn{(1,0),\ (0,3)}$, $H_{5}=\spn{(1,2)}$, $H_{6}=\spn{(1,1)}$, $H_{7}=\spn{(0,1)}$.
Let $Y/X$ be a Galois cover with $\Gal(Y/X)=G$. By \cref{main1}, we have \begin{align}
	\kappa(Y) &= \frac{1}{|G|}\prod_{H\in \scH_{G}}([G:H]\cdot\kappa(X_H))^{-\mu_{\scH_{G}}(\emptyset ,H)}\\ 
	&= \frac{1}{12}\cdot\kappa(X)^{0}\cdot 6\kappa(X_{1}) \cdot 6\kappa(X_{2}) \cdot 6\kappa(X_{3}) \notag \\
	&\quad \cdot (3\kappa(X_{4}))^{-2}\cdot
(2\kappa(X_{5}))^{0}\cdot (2\kappa(X_{6}))^{0}\cdot (2\kappa(X_{7}))^{0}  \\ 
	&= 2\cdot\frac{\kappa(X_{1}) \kappa(X_{2}) \kappa(X_{3})}{{\kappa(X_4)}^2},
\end{align}
where the $X_i$ are corresponding to $H_i$.

	Let $X$ be the bouquet graph with two loops. Take an orientation $S=\{e_1,e_2\}$ of $X$ and define a voltage assignment $\alpha\colon S\to G$ by $$\alpha(e_1)=(1,0),\ \alpha(e_2)=(0,1). $$
Then we have the Galois covering $\pi\colon Y=X(\alpha) \to X$ as in Figure \ref{fig:graph_mapping}. 

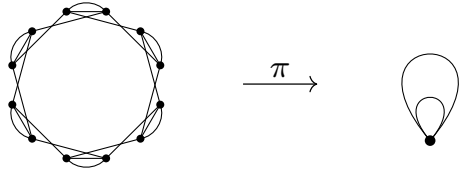
\begin{figure}[h]
\centering
\begin{tikzcd}
  \begin{tikzpicture}[baseline={([yshift=-0.5ex] current bounding box.center)}] 

\node[draw=none, minimum size=2cm ,regular polygon,regular polygon sides=12] (a) {};

\foreach \x in {1,2,...,12}
  \fill (a.corner \x) circle[radius=1.5pt];
  
\foreach \y\z in {1/3,2/4,3/5,4/6,5/7,6/8,7/9,8/10,9/11,10/12,11/1,12/2}
  \path (a.corner \y) edge (a.corner \z);
  
\foreach \y\z in {1/2,3/4,5/6,7/8,9/10,11/12}
  \path (a.corner \y) edge (a.corner \z); 

\path (a.corner 1) edge [bend right=50] (a.corner 2);
\path (a.corner 3) edge [bend right=50] (a.corner 4);
\path (a.corner 5) edge [bend right=50] (a.corner 6);
\path (a.corner 7) edge [bend right=50] (a.corner 8);
\path (a.corner 9) edge [bend right=50] (a.corner 10);
\path (a.corner 11) edge [bend right=50] (a.corner 12);  
   
\end{tikzpicture} \qquad \arrow[r,"\scalebox{1.5}{$\pi$}" ]   &	\begin{tikzpicture} [baseline={([yshift=-0.5ex] current bounding box.center)}, scale=2]

\node[draw=none,minimum size=4cm,regular polygon,regular polygon sides=1] (a) {};

\foreach \x in {1}
  \fill (a.corner \x) circle[radius=1pt];
\draw (a.corner 1) to [in=50,out=130,loop] (a.corner 1);
\draw (a.corner 1) to [in=50,out=130,distance = 0.5cm,loop] (a.corner 1);


\end{tikzpicture}
\end{tikzcd}
\caption{The derived graph $X(\alpha)$ (left) and a bouquet graph $X$ with two loops (right).}
\label{fig:graph_mapping}
\end{figure}

\noindent Moreover, we obtain the Hasse diagram of $\ul{\scH_{G}}$ and the intermediate graphs of $Y/X$ as in Figure \ref{fig:Hasse H}.
\begin{figure}[h]
    \centering
    \begin{tabular}{c@{\hspace{3cm}}c}
      \begin{minipage}{0.25\linewidth}
\begin{tikzcd}[scale=2]
&  & \emptyset \arrow[ld, no head] \arrow[d, no head] \arrow[rd, no head] & \\ 
& H_{1} \arrow[ld, no head] \arrow[rd, no head] & H_{2} \arrow[ld, no head] \arrow[d, no head] \arrow[d, no head] & H_{3} \arrow[ld, no head] \arrow[d, no head]  \\
H_{5} \arrow[rrd, no head] & H_{6} \arrow[rd, no head] & H_{4} \arrow[d, no head] & H_{7} \arrow[ld, no head] \\
& & G &
\end{tikzcd}
\end{minipage} &
\begin{minipage}{0.5\linewidth}
  \begin{tikzcd}[sep=1.5em, font=\small,scale cd=0.7,nodes in empty cells]
 {}  & {}   &  \begin{tikzpicture}[baseline={([yshift=-0.5ex] current bounding box.center)}]

\node[draw=none,minimum size=2cm,regular polygon,regular polygon sides=12] (a) {};

\foreach \x in {1,2,...,12}
  \fill (a.corner \x) circle[radius=1pt];
  
\foreach \y\z in {1/3,2/4,3/5,4/6,5/7,6/8,7/9,8/10,9/11,10/12,11/1,12/2}
  \path (a.corner \y) edge (a.corner \z);
  
\foreach \y\z in {1/2,3/4,5/6,7/8,9/10,11/12}
  \path (a.corner \y) edge (a.corner \z); 

\path (a.corner 1) edge [bend right=50] (a.corner 2);
\path (a.corner 3) edge [bend right=50] (a.corner 4);
\path (a.corner 5) edge [bend right=50] (a.corner 6);
\path (a.corner 7) edge [bend right=50] (a.corner 8);
\path (a.corner 9) edge [bend right=50] (a.corner 10);
\path (a.corner 11) edge [bend right=50] (a.corner 12);  
   
\end{tikzpicture} \arrow[d]  \arrow[ld] \arrow[rd]  & {}    \\  
   & \begin{tikzpicture}

\node[draw=none,minimum size=1.5cm,regular polygon,regular polygon sides=6] (a) {};

  \foreach \x in {1,2,...,6}
  \fill (a.corner \x) circle[radius=1pt];
  
 \foreach \y\z in {1/2,2/3,3/4,4/5,5/6,6/1}
  \path (a.corner \y) edge (a.corner \z);
  
  \draw (a.corner 1) to [distance = 0.7cm,loop] (a.corner 1);
  \draw (a.corner 2) to [distance = 0.7cm,loop] (a.corner 2);
  \draw (a.corner 3) to [distance = 0.7cm,loop] (a.corner 3);
  \draw (a.corner 4) to [distance = 0.7cm,loop] (a.corner 4);
  \draw (a.corner 5) to [distance = 0.7cm,loop] (a.corner 5);
  \draw (a.corner 6) to [distance = 0.7cm,loop] (a.corner 6);

\end{tikzpicture}\arrow[ld]\arrow[rd]   & \begin{tikzpicture} 
\node[draw=none,minimum size=2cm,regular polygon,regular polygon sides=6] (a) {};

\foreach \x in {1,2,...,6}
  \fill (a.corner \x) circle[radius=1pt];

\foreach \y\z in {1/2,1/5,1/3,2/4,3/4,4/5,5/6,3/6,2/6}
  \path (a.corner \y) edge (a.corner \z);
 
\path (a.corner 1) edge [bend right=50] (a.corner 2);
\path (a.corner 3) edge [bend right=50] (a.corner 4);
\path (a.corner 5) edge [bend right=50] (a.corner 6);

\end{tikzpicture}\arrow[ld]\arrow[d]  & \begin{tikzpicture} 
\node[draw=none,minimum size=2cm,regular polygon,regular polygon sides=6] (a) {};

\foreach \x in {1,2,...,6}
  \fill (a.corner \x) circle[radius=1pt];

\foreach \y\z in {1/3,1/5,2/4,2/6,3/5,4/6,1/2,3/4,5/6}
  \path (a.corner \y) edge (a.corner \z);
 
\path (a.corner 1) edge [bend right=50] (a.corner 2);
\path (a.corner 3) edge [bend right=50] (a.corner 4);
\path (a.corner 5) edge [bend right=50] (a.corner 6);

\end{tikzpicture}\arrow[ld]\arrow[d]   \\  
 \begin{tikzpicture}
 
\node[draw=none,minimum size=1.5cm,regular polygon,rotate = -45,regular polygon sides=4] (a) {};

  \fill (a.corner 1) circle[radius=0.7pt];
  \fill (a.corner 3) circle[radius=0.7pt];
  
  \path (a.corner 1) edge [bend left=20] (a.corner 3);
  \path (a.corner 1) edge [bend right=20] (a.corner 3);

  \draw (a.corner 3) to [in=150,out=210, distance = 0.7cm,loop] (a.corner 3);
  \draw (a.corner 1) to [in=30,out=330, distance = 0.7cm,loop] (a.corner 1);

\end{tikzpicture}\arrow[rrd]  &  \begin{tikzpicture}

\node[draw=none,minimum size=2cm,regular polygon,rotate = -45,regular polygon sides=4] (a) {};

  \fill (a.corner 1) circle[radius=0.7pt];
  \fill (a.corner 3) circle[radius=0.7pt];
  
  \path (a.corner 1) edge [bend left=20] (a.corner 3);
  \path (a.corner 1) edge [bend left=60] (a.corner 3);
  \path (a.corner 1) edge [bend right=20] (a.corner 3);
  \path (a.corner 1) edge [bend right=60] (a.corner 3);

\end{tikzpicture}\arrow[rd]  & \begin{tikzpicture}

\node[draw=none,minimum size=2cm,regular polygon,regular polygon sides=3] (a) {};

\foreach \x in {1,2,...,3}
  \fill (a.corner \x) circle[radius=1pt];
  
\foreach \y\z in {1/2,2/3,3/1}
  \path (a.corner \y) edge (a.corner \z);
  \draw (a.corner 1) to [in=50,out=130, distance = 0.7cm,loop] (a.corner 1);
  \draw (a.corner 2) to [in=170,out=250, distance = 0.7cm,loop] (a.corner 2);
  \draw (a.corner 3) to [in=10,out=290, distance = 0.7cm,loop] (a.corner 3);
 
\end{tikzpicture}\arrow[d]   & \begin{tikzpicture}

\node[draw=none,minimum size=1.5cm,regular polygon,rotate = -45,regular polygon sides=4] (a) {};

  \fill (a.corner 1) circle[radius=0.7pt];
  \fill (a.corner 3) circle[radius=0.7pt];
  
  \path (a.corner 1) edge [bend left=20] (a.corner 3);
  \path (a.corner 1) edge [bend right=20] (a.corner 3);

  \draw (a.corner 3) to [in=150,out=210, distance = 0.7cm,loop] (a.corner 3);
  \draw (a.corner 1) to [in=30,out=330, distance = 0.7cm,loop] (a.corner 1);

\end{tikzpicture}\arrow[ld]    \\ 
 {}  &  {}  &  \begin{tikzpicture} [baseline={([yshift=-0.6ex] current bounding box.center)}]
  
\node[draw=none,minimum size=2cm,regular polygon,regular polygon sides=1] (a) {};

\foreach \x in {1}
  \fill (a.corner \x) circle[radius=0.7pt];
\draw (a.corner 1) to [in=50,out=130,loop] (a.corner 1);
\draw (a.corner 1) to [in=50,out=130,distance = 0.5cm,loop] (a.corner 1);
\end{tikzpicture} &  {}     	
\end{tikzcd}
\end{minipage}
\end{tabular}
\caption{The Hasse diagram of $\ul{\scH_{G}}$ (left) and the intermediate graphs of $Y/X$ (right).}\label{fig:Hasse H}

\end{figure}

Using SageMath \cite{Sage}, we calculate 
$$ \kappa(X_{1})=6,\ \kappa(X_{2})=300,\ \kappa(X_{3})=294,\ \kappa(X_4)=3.$$
By the spanning tree formula above, we obtain 
$$\kappa(Y) = 2\cdot\frac{\kappa(X_{1}) \kappa(X_{2})\kappa(X_{3})}{{\kappa(X_{4})}^2} =117600. $$	
\end{eg}

\subsection{The Brauer--Kuroda relations for graphs }\label{bkru}
In this section, we prove \cref{main2} and consider several examples.

\begin{lem}\label{inter rel}
	Let $Y/X$ be a Galois covering with $\Gal(Y/X)=G$, and let $H$ be a subgroup of $G$. Let $a_{\rho,H}$ be the integers satisfying $\displaystyle{\chi_{\Ind^{G}_{H}(\rho_{0})}=\chi_{\rho_{0}}+\sum_{\rho\in \Irr(G)\setminus\{\rho_0\}} a_{\rho,H}\chi_{\rho}}$, then 
\[
[G:H]\cdot\kappa(X_H)=\kappa(X)\cdot\prod_{\rho\in \Irr(G)\setminus\{\rho_0\}} h_{Y/X}(1,\rho)^{a_{\rho,H}}.
\]
\end{lem}
	
\begin{proof}
We have 
\begin{align}
	\zeta_{X_{H}}(u) & \overset{\phantom{\ref{artinL pro}(3)}}{=} L_{Y/X_{H}}(u,\rho_{0})\\ & \overset{\text{\ref{artinL pro}(3)}}{=} L_{Y/X}(u,\Ind^{G}_{H}(\rho_0))\\ &\overset{\text{\ref{artinL pro}(1)}}{=}\zeta_{X}(u)\cdot \prod_{\rho\in \Irr(G)\setminus\{\rho_{0}\}} L_{Y/X}(u,\rho)^{a_{\rho,H}}. \end{align} 
Since $\deg(\Ind_{H}^{G}(\rho_{0}))=[G:H]$ and $\chi(X_{H})=[G:H]\cdot\chi(X)$, we have $$ h_{X_{H}}(u)=h_{X}(u)\cdot\prod_{\rho\in \Irr(G)\setminus\{\rho_{0}\}} h_{Y/X}(u,\rho)^{a_{\rho,H}}.$$
By differentiating both sides of this equation and substituting $u=1$, we have the claimed equation. 
\end{proof}


\begin{dfn}\label{dfn:cg}
For a finite group $G$, we define 
\[
\scC_{G} \coloneq \{C  \mid C \text{ is a cyclic subgroup of } G\}.
\]
Then the set $\scC_{G}$ becomes a partially ordered set by the inclusion relation. We also define a partially ordered set 
\[
\ol{\scC_{G}} \coloneq \scC_{G}\cup \{\infty\},
\]
where we define $C<\infty$ for any $C\in \scC_{G}$.
\end{dfn}


\begin{thm}[Artin's induction theorem, cf.~{\cite[Theorem 2.1.3]{Brauer}}]\label{artinn}
 Let $\chi$ be a rational-valued character of $G$. For every cyclic subgroup $C$ of $G$, let 
\[
a_{\chi}(C)\coloneq \frac{1}{[G:C]}\sum_{\substack{B\in \scC_{G}; \\ C\subseteq B}}\mu([B:C])\chi(g_B),
\] 
where $g_{B}$ is a generator of $B$ and $\mu$ is the classical M\"{o}bius function. Then we have $$\chi=\sum_{\substack{C\in \scC_{G}}}a_{\chi}(C)\chi_{\Ind^{G}_{C}(\rho_0)}. $$ 
\end{thm}

\begin{cor}\label{triv}
	In the above notation, we have $$\chi_{\rho_{0}}=\sum_{C\in \scC_{G}}-\frac{1}{[G:C]}\mu(C,\infty)\chi_{\Ind_{C}^{G}(\rho_{0})} . $$
\end{cor}

\begin{proof}

In \cref{artinn}, setting $\chi$ to be the trivial character $\chi_{\rho_{0}}$ of $G$, we get

$$\chi_{\rho_{0}}=\sum_{C \in \scC_{G}} a(C) \chi_{\Ind^{G}_{C}(\rho_0)},$$
where $a(C)$ is given by

$$a(C) = \frac{1}{[G:C]} \sum_{\substack{B\in \scC_{G}; \\ C \subseteq B}} \mu([B : C]).$$
By Lemma \ref{grp lat}, we can simplify the sum over $C \subseteq B$ to obtain

$$a(C) = \frac{1}{[G:C]} \sum_{\substack{B\in \scC_{G}; \\ C \subseteq B}} \mu(C, B) = -\frac{1}{[G:C]} \mu(C,\infty).$$
Therefore, we conclude that

$$\chi_{\rho_{0}} = \sum_{C \in \scC} -\frac{1}{[G:C]} \mu(C, \infty) \chi_{\Ind^{G}_{C}(\rho_{0})}.$$

\end{proof}

\begin{thm}\label{main2}
Let $G$ be a finite group, and let $\ol{\scC_{G}}$ be the partially ordered set defined in Definition \ref{dfn:cg}. Let $\mu\colon \ol{\scC_{G}}\times\ol{\scC_{G}}\to \bbZ$ be the M\"{o}bius function of $\ol{\scC_{G}}$ (cf.~Appendix \ref{mobi fu}). Then, for any Galois cover $Y/X$ with $\Gal(Y/X)=G$, we have   
\[
\kappa(X)=\prod_{ C\in \scC_{G}}([G:C]\cdot \kappa(X_C))^{-\frac{1}{[G:C]} \mu(C,\infty)},
\] 
where $X_C$ denotes the intermediate graph corresponding to $C$.
\end{thm}

\begin{proof}

If $\chi(X)= 0$, then $\Gal(Y/X)$ is cyclic (see Remark 2.7 in \cite{ray}). Then $\scC_{G} $ contains $G$, and hence, we have $\mu(C,\infty) = 0$ for any $C \in \scC_{G}\setminus\{G\}$. In this case, the theorem is true. From now on, we assume $\chi(X)\neq 0$.

Let $a_{\rho,C}$ be the integers satisfying $\displaystyle{\chi_{\Ind^{G}_{H}(\rho_{0})}=\chi_{\rho_{0}}+\sum_{\rho\in \Irr(G)\setminus\{\rho_0\}} a_{\rho,C}\chi_{\rho}}$ for each $C\in \scC_{G}$. By Corollary \ref{triv}, we have \begin{align}
	\chi_{\rho_0} &\overset{\text{\ref{inter rel}}}{=} \sum_{C\in \scC_{G}}\left(-\frac{1}{[G:C]}\mu(C,\infty)\cdot\left(\chi_{\rho_{0}}+\sum_{\rho\in \Irr(G)\setminus\{\rho_0\} } a_{\rho,C}\chi_{\rho}\right)\right) \\ &= \left(\sum_{C\in \scC_{G}}-\frac{1}{[G:C]}\mu(C,\infty)\right)\chi_{\rho_0}+ \sum_{\rho\in \Irr(G)\setminus\{\rho_{0}\}}\left(-\sum_{C\in \scC_{G}} \frac{1}{[G:C]} \mu(C,\infty) a_{\rho,C}\right)\chi_{\rho}. 
\end{align}
Hence we obtain 
 \begin{equation}
 	\text{$-\sum_{ C\in \scC_{G}}\frac{1}{[G:C]}\mu(C,\infty)=1$\qquad and\qquad $\sum_{C\in \scC_{G}}\frac{1}{[G:C]} \mu(C,\infty)a_{\rho,C}=0$} \tag{3}\label{keylem}
 \end{equation}
for any $\rho\in \Irr(G)\setminus\{\rho_0\} $.
Then we have 
\begin{align}
	\prod_{C\in \scC_{G}}([G:C] \cdot \kappa(X_C)&)^{-\frac{1}{[G:C]}\mu(C,\infty)} \\& \overset{\text{\ref{inter rel}}}{=} \prod_{C\in \scC_{G}}\left(\kappa(X) \cdot \prod_{\rho\in \Irr(G)\setminus\{\rho_0 \}} h_{Y/X}(1,\rho)^{a_{\rho,C}}\right)^{-\frac{1}{[G:C]}\mu(C,\infty)} \\ &= \displaystyle{\kappa(X)^{-\sum\limits_{C\in \scC_{G}} \frac{1}{[G:C]} \mu(C,\infty)}\cdot\!\!\!\!\!\!\prod_{\rho\in \Irr(G)\setminus\{\rho_0\}}\hspace{-15pt} h_{Y/X}(1,\rho)^{-\sum\limits_{C\in \scC_{G}}\frac{1}{[G:C]} \mu(C,\infty) a_{\rho,C}}} \\ & \overset{\eqref{keylem}}{=}\kappa(X).
\end{align}
\end{proof}


\begin{rmk}\ \label{rmk:excrem}
\begin{enumerate}
	\item If $G$ is cyclic, then \cref{main2} implies the trivial equation $\kappa(X)=\kappa(X)$. Conversely, if $G$ is not cyclic, then \cref{main2} gives a non-trivial formula since $\kappa(X)$ does not appear on the right-hand side of the equation.
	\item In \cref{main2}, if $\mu(\{1_G\},\infty)=0$, then $\kappa(Y)$ does not appear on the right-hand side of the equation. For example, cyclic groups satisfy the condition $\mu(\{1_G\},\infty) = 0$.  In this case, \cref{main2} does not give a formula expressing $\kappa(Y)$. Thus, \cref{main2} is not useful to represent $\kappa(Y)$ by the number of spanning trees of intermediate graphs. However, as we will see in Example \ref{eg:q8}, \cref{main2} may give a non-trivial formula even in this case.

In \cite{excep}, a group $G$ with the condition $\mu(\{1_{G}\},\infty)=0$ is called \tit{exceptional}. In Table \ref{figure}, we give the list of groups with order at most $24$, and we indicate whether they are exceptional or not. We computed Table \ref{figure} using data from \cite{GroupNames}.
\end{enumerate}
\end{rmk}


\begin{rmk}
Let $G$ be a finite group. A relation of the form 
\[
\sum_{H} n_{H} \chi_{\Ind_{H}^{G}(\rho_0)} = 0,
\]
is called a \tit{Brauer relation}. Artin's induction theorem (\cref{artinn}) is an example of a Brauer relation. Many Brauer relations are known, and Brauer relations for various groups are presented in \cite{BrauerRel}. If $G$ has a Brauer relation of the above form, by a similar argument to \cref{main2}, we can prove a spanning tree formula
\[
\prod_{H} ([G : H] \cdot \kappa(X_{H}))^{n_{H}} = 1.
\] 
\end{rmk}


\begin{eg}\label{bru val}
 Let $Y/X$ be a Galois cover with Galois group $G=(\bbZ/\bbZ)^m$. To see that \cref{main2} is a generalization of \cref{thm:HMSV}, we use induction on $m$. $G$ has $2^{m}-1$ cyclic subgroups $C_{1},\ldots,C_{2^{m}-1}$ of order $2$, and we have $$\ol{\scC_{G}}=\{\{1_G\},C_{1},\ldots,C_{2^{m}-1},\infty \}.$$ Moreover, we have $\mu(\{1_{G}\},\infty)=2^{m}-2$ and $\mu(C_{i},\infty)=-1$ for $i=1,\ldots,2^m-1$. Therefore, by \cref{main2}, we obtain $$
    \kappa(X) = (2^{m}\kappa(Y))^{-\frac{2^m-2}{2^m}} \left(\prod_{i=1}^{2^m-1}2^{m-1}\kappa(X_{C_{i}})\right)^{\frac{1}{2^{m-1}}},$$ which simplifies to \begin{equation}\label{eq44}
    	\kappa(Y)^{2^{m-1}-1}=\frac{2^{m2^{m-1}-2^{m}+1}}{\kappa(X)^{2^{m-1}}} \prod_{i=1}^{2^m-1}\kappa(X_{C_{i}}).\tag{4}
    \end{equation} 
    Since each $X_{C_{i}}$ is a $(\bbZ/2\bbZ)^{m-1}$-cover, by induction, we have\begin{equation}\label{eq55}
    	\kappa(X_{C_{i}})=\frac{2^{2^{m-1}-m}}{\kappa(X)^{2^{m-1}-2}}\prod_{j=1}^{2^{m-1}-1}\kappa(X_{H_{ij}}),\tag{5}
    \end{equation}  where each $H_{ij}$ is a subgroup of $G$ with index $2$ that contains $C_{i}$. Combining \eqref{eq44} with \eqref{eq55}, we have
 
 \begin{align}
 	\kappa(Y)^{2^{m-1}-1} &\overset{\eqref{eq44}}{=} \frac{2^{m2^{m-1}-2^{m}+1}}{\kappa(X)^{2^{m-1}}} \prod_{i=1}^{2^m-1}\kappa(X_{C_{i}})\\
 	                      &\overset{\eqref{eq55}}{=} \frac{2^{m2^{m-1}-2^{m}+1}}{\kappa(X)^{2^{m-1}}} \prod_{i=1}^{2^m-1} \left( \frac{2^{2^{m-1}-m}}{\kappa(X)^{2^{m-1}-2}}\prod_{j=1}^{2^{m-1}-1}\kappa(X_{H_{ij}}) \right)\\
 	                      &= \left(\frac{2^{2^{m}-m-1}}{\kappa(X)^{2^{m}-2}}\prod_{i=1}^{2^{m}-1}\kappa(X_{H_i}) \right)^{2^{m-1}-1 },
 \end{align} where $H_{1},\ldots,H_{2^m-1}$ are subgroups of $G$ with index $2$. Hence we conclude that
	$$\kappa(Y)=\frac{2^{2^{m}-m-1}}{\kappa(X)^{2^{m}-2}}\prod_{i=1}^{2^{m}-1}\kappa(X_{H_i}).$$	
	Thus, \cref{main2} leads to the formula in \cref{thm:HMSV} by induction. Therefore, \cref{main2} can be seen as a generalization of \cref{thm:HMSV}.
\end{eg}


\begin{eg}\label{eg:sym}
	Let $G=S_3$ be the symmetric group of degree 3. $G$ has five cyclic subgroups 
\[
C_{1}=\{1_G\},\ C_{2}=\{1_G,(12)\},\ C_{3}=\{1_G,(13)\},\ C_{4}=\{1_G,(23) \},\ C_{5}=\{1_G,(123),(132) \}.
\]
Let $Y/X$ be a Galois cover with $\Gal(Y/X)=G$. By \cref{main2}, we have 
\[
\kappa(X)=(6\kappa(Y))^{-\frac{1}{2}}\cdot(3\kappa(X_{2}))^{\frac{1}{3}}\cdot (3\kappa(X_{3}))^{\frac{1}{3}}\cdot (3\kappa(X_{4}))^{\frac{1}{3}}\cdot(2\kappa(X_5))^{\frac{1}{2}}, 
\]
where the $X_i$ are corresponding to $C_{i}$. 
Since $C_{2},C_{3},$ and $C_{4}$ are conjugate, it follows that $\kappa(X_{2})=\kappa(X_{3})=\kappa(X_{4})$. Then we obtain \begin{equation}
	\kappa(Y)= \frac{3\kappa(X_2) \kappa(X_5)^{2}}{\kappa(X)^{2}}.\tag{6} \label{abfor}
\end{equation}

Let $X$ be the bouquet graph with two loops. Take an orientation $S=\{e_1,e_2\}$ of $X$ and define a voltage assignment $\alpha\colon S\to G$ by $$\alpha(e_1)=(12),\ \alpha(e_2)=(123). $$
Then we obtain the Hasse diagram of $\ol{\scC_{G}}$ and the intermediate graphs of $Y/X$ as in Figure \ref{fig:S3}.

\begin{figure}[htbp]
\centering
\begin{tabular}{cc}
\begin{minipage}{0.45\linewidth}
\centering
\begin{tikzcd}
{} & \infty \arrow[rd,dash] \arrow[ld,dash] & {} \\
C_{2} \arrow[rd,dash] & {} & C_{5} \arrow[ld,dash] \\
{} & \{1_G\} & {}
\end{tikzcd}
\end{minipage}
&
\begin{minipage}{0.45\linewidth}
\centering
\begin{tikzcd}[sep=1.8em, font=\small,scale cd=0.8,nodes in empty cells]
{} & \begin{tikzpicture}[baseline={([yshift=-0.6ex] current bounding box.center)}]
\node[draw=none,minimum size=2cm,regular polygon,regular polygon sides=6] (a) {};
\fill (a.corner 1) circle[radius=0.7pt];
\fill (a.corner 3) circle[radius=0.7pt];
\foreach \y\z in {1/6,1/5,2/3,2/4,3/4,5/6}
\path (a.corner \y) edge (a.corner \z);
\path (a.corner 1) edge [bend left=20] (a.corner 2);
\path (a.corner 1) edge [bend right=20] (a.corner 2);
\path (a.corner 3) edge [bend left=20] (a.corner 5);
\path (a.corner 3) edge [bend right=20] (a.corner 5);
\path (a.corner 4) edge [bend left=20] (a.corner 6);
\path (a.corner 4) edge [bend right=20] (a.corner 6);
\end{tikzpicture} \arrow[dl] \arrow[dr] & {} \\
\begin{tikzpicture}[baseline={([yshift=-0.6ex] current bounding box.center)}]
\node[draw=none,minimum size=2cm,regular polygon,regular polygon sides=3] (a) {};
\fill (a.corner 1) circle[radius=0.7pt];
\fill (a.corner 3) circle[radius=0.7pt];
\foreach \y\z in {1/2,2/3,3/1}
\path (a.corner \y) edge (a.corner \z);
\path (a.corner 2) edge [bend left=40] (a.corner 3);
\path (a.corner 2) edge [bend right=40] (a.corner 3);
\draw (a.corner 1) to [in=50,out=130,distance=1cm,loop] (a.corner 1);
\end{tikzpicture} \arrow[dr]
& {} & \begin{tikzpicture}[baseline={([yshift=-0.6ex] current bounding box.center)}]
\node[draw=none,minimum size=1.5cm,regular polygon,rotate=-45,regular polygon sides=4] (a) {};
\fill (a.corner 1) circle[radius=0.7pt];
\fill (a.corner 3) circle[radius=0.7pt];
\path (a.corner 1) edge [bend left=40] (a.corner 3);
\path (a.corner 1) edge [bend right=40] (a.corner 3);
\draw (a.corner 3) to [in=150,out=210, distance=1cm,loop] (a.corner 3);
\draw (a.corner 1) to [in=30,out=330, distance=1cm,loop] (a.corner 1);
\end{tikzpicture} \arrow[ld] \\
{} & \begin{tikzpicture}[baseline={([yshift=-1.7ex] current bounding box.center)}]
\node[draw=none,minimum size=2cm,regular polygon,regular polygon sides=1] (a) {};
\foreach \x in {1}
\fill (a.corner \x) circle[radius=0.7pt];
\draw (a.corner 1) to [in=50,out=130, distance=1cm,loop] (a.corner 1);
\draw (a.corner 1) to [in=50,out=130,distance=0.5cm,loop] (a.corner 1);
\end{tikzpicture}
& {}
\end{tikzcd}
\end{minipage}
\end{tabular}

\caption{The Hasse diagram of $\ol{\scC_{G}}$ (left) and the intermediate graphs of $Y/X$ (right).}

\label{fig:S3}
\end{figure}

\noindent Using SageMath \cite{Sage}, we calculate$$\kappa(X_2)=2,\ \kappa(X_5)=7,\ \kappa(X)=1.$$ By the above formula (\ref{abfor}), we have
$$\kappa(Y)=\frac{3\kappa(X_2) \kappa(X_5)^{2}}{\kappa(X)^{2}}=294. $$
\end{eg}


\begin{eg}\label{eg:q8}
 Let $G=Q_{8}=\{\pm 1,\pm \textbf{i},\pm \textbf{j},\pm \textbf{k}\}$ be the quaternion group. Then $G$ is exceptional. $G$ has five cyclic subgroups 
\[
C_1=\{1\},\ C_{2}=\{\pm 1\},\ C_{3}=\{\pm 1,\pm \textbf{i}\},\ C_{4}=\{\pm 1,\pm \textbf{j}\},\ C_{5}=\{\pm 1,\pm \textbf{k}\}.
\] By \cref{main2}, we have 
\[
\kappa(X)=(8\kappa(Y))^{0}\cdot(4\kappa(X_{2}))^{-\frac{1}{2}}\cdot(2\kappa(X_{3}))^{\frac{1}{2}}\cdot(2\kappa(X_{4}))^{\frac{1}{2}}\cdot(2\kappa(X_{5}))^{\frac{1}{2}},
\]
and hence, we have 
\[
\kappa(X_{2})\kappa(X)^2=2\kappa(X_{3})\kappa(X_{4})\kappa(X_{5}).
\]
\end{eg}


\subsection{Non-existence of a spanning tree formula for cyclic Galois groups}\label{exist}

Finally, we consider the case where $G$ is a cyclic group. In this section, we prove \cref{main3}. 


\begin{dfn}\   Let $\bfp=(p_{1},\ldots,p_{\ell})$, $\bfa=(a_{1},\ldots,a_{\ell})$, $\bfb=(b_{1},\ldots,b_{\ell})\in \bbZ^{\ell}$.
	\begin{enumerate}  
		\item We define $$\bfp^{\bfa}\coloneq p_{1}^{a_{1}}\cdots p_{\ell}^{a_{\ell}}.$$
		\item The \tit{join} $\bfa\vee\bfb$ of $\bfa$ and $\bfb$ is defined by $$\bfa\vee\bfb\coloneq (\max\{a_{i},b_{i}\})_{i}.$$
		\item The \tit{meet} $\bfa\wedge\bfb$ of $\bfa$ and $\bfb$ is defined by $$\bfa\wedge\bfb\coloneq (\min \{a_{i},b_{i}\})_{i}.$$
		\item We define the relation $\bfa<\bfb$ on $\bbZ^{\ell}$ if there exists $1\leq i\leq \ell$ such that $a_{1}=b_{1},\cdots, a_{i-1}=b_{i-1}$ and $a_{i}<b_{i}$.  
	\end{enumerate}
\end{dfn} 


\begin{lem}[cf.~{\cite[Corollary 5.3]{Val3}}]\label{arcyclic}
	Let $G$ be a finite abelian group and $X$ be the bouquet graph. Let $\alpha\colon S\to G$ be a voltage assignment satisfying that $X(\alpha)$ is connected. If $\rho\in \Irr(G)$, then we have $$h_{X(\alpha)/X}(1,\rho)=\sum_{s\in S}(2-\rho(\alpha(s))-\rho(-\alpha(s))). $$
\end{lem}


\begin{lem}\label{deg}
Let $\bfp$, $\bfa\in \bbN^{\ell}$. Let $X$ be the bouquet graph with $t+1$ loops and let $G=\bbZ/\bfp^{\bfa}\bbZ$. Take an orientation $S=\{e_1, e_2,\ldots,e_{t+1}\}$ of $X$ and define a voltage assignment $\alpha\colon S\to G$ by $$\alpha(e_1)=\cdots=\alpha(e_{t})=\bfp^{\bfb}, \quad \alpha(e_{t+1})=1,$$ for $\bfb\in \bbN^{\ell}$. Then we have the following.  \begin{enumerate}
	\item $\kappa(X(\alpha))$ is a polynomial in $t$.
	\item The degree $\deg(\kappa(X(\alpha)))$ of $\kappa(X(\alpha))$ equals $\displaystyle{ \bfp^{\bfa}\cdot \left(1-\frac{1}{\bfp^{(\bfa-\bfb) \vee \bm{0}}}\right)}$.
\end{enumerate}
\end{lem}


\begin{proof}	

Let $i=\bfp^{\bfa}$ and $j=\bfp^{\bfb}$. By Proposition \ref{formula} and Lemma \ref{arcyclic}, we have
\begin{align}
	\kappa(X(\alpha)) &\overset{\text{\ref{formula}}}{=}\frac{\kappa(X)}{|G|}\cdot\prod_{\rho\in \Irr(G)\setminus \rho_{0}}h_{X(\alpha)/X}(1,\rho)\\ 
	&\overset{\text{\ref{arcyclic}}}{=} \frac{1}{i} \prod_{k=1}^{i-1} \left( (2-\zeta_{i}^{k}-\zeta_{i}^{-k})+t(2-\zeta_{i}^{k j}-\zeta_{i}^{-k j}) \right)\\
	&= \frac{1}{i}\prod_{k=1}^{i-1}\left(4\sin^2\frac{\pi k}{i}+4t\sin^2\frac{\pi k j}{i} \right),
\end{align}
where $\displaystyle{\zeta_{i}=\cos\frac{2\pi}{i}+\sqrt{-1}\sin\frac{2\pi}{i}}$, which is an $i$-th root of unity. Thus $\kappa(X(\alpha))$ is a polynomial in $t$, and the degree of $\kappa(X(\alpha))$ is obtained by 
\begin{align}
\deg(\kappa(X(\alpha))) 
&=(i-1)-|\{k\mid 1\leq k \leq i-1,\ \sin\frac{\pi k j}{i}=0 \}| \\
&=(i-1)-|\{k\mid1\leq k \leq i-1,\ i | k j \}| \\
&=(\bfp^{\bfa}-1)-(\bfp^{\bfa\wedge\bfb}-1)\\
&= \bfp^{\bfa}\cdot(1-\bfp^{(\bfa\wedge\bfb)-\bfa})\\
&= \bfp^{\bfa}\cdot\left(1-\frac{1}{\bfp^{(\bfa-\bfb)\vee\bm{0}}}\right).
\end{align}
\end{proof}
For a matrix $A$, let $A(m_1,m_2)$ denote the submatrix of $A$ with row $m_1$ and column $m_2$ deleted. If $m_1=m_2=m$, then we will write $A(m_1,m_2)$ as $A(m)$.  


\begin{lem}[Cauchy--Binet Formula, cf.~{\cite[p.214]{BJWG}}]\label{CB}
	Let $A$ and $B$ be $N\times N$ matrices. Then, for any $m_1$, $m_2$ with $1\leq m_1\leq N$ and $1\leq m_2\leq N$, we have
	
$$\det(AB(m_1,m_2))=\sum_{1\leq m\leq N}(\det (A(m_1,m)))(\det (B(m,m_2))).$$ 
	
\end{lem}

If $P$ is an $m\times n$ matrix and $Q$ is a $p\times q$ matrix, then the \tit{Kronecker product} $P\otimes Q$ is defined by the $pm\times qn$ block matrix $$P\otimes Q= \begin{pmatrix} 
  p_{11}Q & p_{12}Q & \dots  & p_{1n}Q \\
  p_{21}Q & p_{22}Q & \dots & p_{2n}Q \\
  \vdots & \vdots & \ddots & \vdots \\
  p_{m1}Q & p_{m2}Q & \dots & p_{mn}Q
\end{pmatrix}. $$


\begin{lem}\label{matrix}
Let $\bfp=(p_{1},\ldots,p_{\ell})$, $\bfs=(s_{1},\ldots,s_{\ell}) \in \bbN^{\ell}$. The matrix $M$ is defined by $$M\coloneq \left(1-\frac{1}{\bfp^{(\bfa+\bfb-\bfs) \vee \bm{0}}}\right)_{\bm{0}<\bfa,\bfb\leq \bfs} \in \bbM_{T-1}(\bbQ),$$ where $\bfa,\bfb\in [0,s_{1}]\times\cdots\times[0,s_{\ell}] $ and $\displaystyle{T \coloneq \prod_{k=1}^{\ell}(s_{k}+1)}$.

Then the determinant of $M$ is given by $$\det(M)= (-1)^{T-1}\prod_{i=1}^{\ell}\left(\frac{1}{p_{i}}-1\right)^{\frac{s_{i}}{s_{i}+1}T}. $$ In particular, we have $\det(M)\neq 0$.

\end{lem}

\begin{proof}
  Consider the matrix $$\ol{M}\coloneq \left(1-\frac{1}{\bfp^{(\bfa+\bfb-\bfs) \vee \bm{0}}}\right)_{\bm{0}\leq\bfa,\bfb\leq \bfs}\in \bbM_{T}(\bbQ).$$
Then we obtain 
\begin{equation}\label{5}
	 \ol{M}=(J_{1}\otimes\cdots\otimes J_{{\ell}})-(K_{1}\otimes\cdots\otimes K_{\ell}), \tag{7}
\end{equation} where
$$J_{i}\coloneq \begin{pmatrix} 
  1 & 1 & \dots & 1 & 1 \\
  1 & 1 & \ddots & 1 & 1 \\
  \vdots & \vdots & \ddots &\vdots & \vdots \\
  1 & 1 & \dots & 1 & 1 \\
  1 & 1 & \dots & 1 & 1
\end{pmatrix},\ K_i \coloneq \begin{pmatrix} 
  1 & 1 & \dots & 1 & 1 \\
  1 & 1 & \dots & 1 & 1/p_i \\
  \vdots & \vdots & \iddots & \iddots & \vdots \\ 
  1 & 1 & \iddots & 1/p_i^{s_i - 2} & 1/p_i^{s_i - 1} \\
  1 & 1/p_i & \dots & 1/p_i^{s_i - 1} & 1/p_i^{s_i} 
\end{pmatrix} \in \bbM_{s_{i}+1}(\bbQ), $$ since, for $\bm{0}\leq \bfa,\bfb\leq \bfs$, we have \[
{(K_i)}_{a_i b_i} = 
\begin{cases}
\hfill 1 \hfill & \text{if } a_i + b_i - s_i \leq 0, \\
1/p_i^{a_i + b_i - s_i} & \text{if } a_i + b_i - s_i \geq 0.
\end{cases}
\]
\noindent Let 
\[
L_{i}\coloneq \begin{pmatrix}
1 & 0 & \dots & 0 & 0 \\
-1 & 1 & \dots & 0 & 0 \\
\vdots & \vdots & \ddots & \vdots & \vdots \\
-1 & 0 & \dots & 1 & 0 \\
-1 & 0 & \dots & 0 & 1
\end{pmatrix},\
R_{i}\coloneq \begin{pmatrix}
1 & -1 & \dots & -1 & -1 \\
0 & 1 & \dots & 0 & 0 \\
\vdots & \vdots & \ddots & \vdots & \vdots \\
0 & 0 & \dots & 1 & 0 \\
0 & 0 & \dots & 0 & 1
\end{pmatrix} \in \bbM_{s_{i}+1}(\bbZ).
\]


\noindent Then we have $$J^{\prime}_{i}\coloneq L_{i}J_{i}R_{i}= \begin{pmatrix}
1 & 0 & \dots & 0 & 0 \\
0 & 0 & \dots & 0 & 0 \\
\vdots & \vdots & \ddots & \vdots & \vdots \\
0 & 0 & \dots & 0 & 0 \\
0 & 0 & \dots & 0 & 0
\end{pmatrix},$$
$$K'_{i} \coloneq L_{i} K_{i} R_{i} = \begin{pmatrix}
1 & 0 & \dots & 0 & 0 \\
0 & 0 & \dots & 0 & 1/p_i - 1 \\
\vdots & \vdots & \iddots & \iddots & \vdots \\
0 & 0 & \iddots & 1/p_i^{s_i - 2} - 1 & 1/p_i^{s_i - 1} - 1 \\
0 & 1/p_i - 1 & \dots & 1/p_i^{s_i - 1} - 1 & 1/p_i^{s_i} - 1
\end{pmatrix}. $$
\noindent Since all entries in the first column and the first row of matrix $\ol{M}$ equal $0$, we have \begin{equation}\label{6}
	 \det (\ol{M}(m_1.m_2))=0 \tag{8}
\end{equation} if $m_1\neq 1$ or $m_2\neq 1$.  

\noindent Hence we conclude that
 \begin{align}
		\det(M) &=\det(\ol{M}(1))\\
		        &\overset{\eqref{6}}{=}\hspace{-15pt} \sum_{1\leq m_1,m_2\leq T}\hspace{-10pt} \det ((L_1\otimes \cdots \otimes L_{\ell})(1,m_1))\cdot \det (\ol{M}(m_1.m_2))\cdot \det((R_1\otimes \cdots \otimes R_{\ell})(m_2,1))\\
		        &\overset{\text{\ref{CB}}}{=}\det(((L_1\otimes \cdots \otimes L_{\ell})\cdot \ol{M} \cdot(R_1\otimes \cdots \otimes R_{\ell}))(1))\\
		        &\overset{\eqref{5}}{=}\det(((J^{\prime}_{1}\otimes\cdots\otimes J^{\prime}_{\ell})-(K^{\prime}_{1}\otimes\cdots\otimes K^{\prime}_{\ell}))(1))\\
		        &=\det(-(K^{\prime}_{1}\otimes\cdots\otimes K^{\prime}_{\ell})(1))\\
		        &=(-1)^{T-1}\det((K^{\prime}_{1}\otimes\cdots\otimes K^{\prime}_{\ell})(1))\\
		        &= (-1)^{T-1}\det((K^{\prime}_{1}\otimes\cdots\otimes K^{\prime}_{\ell}))\\
		       &=(-1)^{T-1}\prod_{i=1}^{\ell}\left(\frac{1}{p_{i}}-1\right)^{\frac{s_{i}}{s_{i}+1}T}.
	\end{align}
\end{proof}
\begin{thm}\label{main3} Let $G=\bbZ/n\bbZ$ be a nontrivial cyclic group, and let $H_{1},\ldots,H_{k}$ be all of its subgroups. Then there do not exist $\bm{0}\neq (m_1,\ldots,m_{k})\in \bbZ^{k}$ and $q\in \bbQ$ satisfying the following condition: For any Galois cover $Y/X$ with $\Gal(Y/X)=G$, it follows that $$q\cdot \prod_{i=1}^{k}\kappa(X_{i})^{m_{i}}=1,$$ where $X_{i}$ is the intermediate graph corresponding to $H_{i}$. 
\end{thm}
\begin{proof}
Suppose the assertion of the theorem is false. Since $G$ is cyclic, we can write $n$ in the form $n=\bfp^{\bfs}$, where $\bfs\in \bbN^{\ell}$ and $p_{1},\ldots,p_{\ell}$ are the prime factors of $n$. Then every subgroup of $G$ is written as $\bfp^{\bfa}\bbZ/\bfp^{\bfs}\bbZ$ for some $\bfa\in [0,s_{1}]\times\cdots\times[0,s_{\ell}]$. Let $X$ be the bouquet graph with $t+1$ loops, and let $S=\{e_1, e_2,\ldots,e_{t+1}\}$ be an orientation of $X$. For $\bfb\in [0,s_{1}]\times\cdots\times[0,s_{\ell}]$, define a voltage assignment $\alpha\colon S\to G$ by $$\alpha(e_1)=\cdots=\alpha(e_{t})=\bfp^{\bfs-\bfb}, \quad \alpha(e_{t+1})=1.$$
For $\bfa \in [0,s_{1}]\times\cdots\times[0,s_{\ell}] $, we consider the voltage assignment $$\alpha_{\bfa}\colon S\to \bbZ/\bfp^{\bfa}\bbZ $$ defined by the composition $$S \xrightarrow{\alpha} G\to \bbZ/\bfp^{\bfa}\bbZ .$$
Since an intermediate graph of $X(\alpha)/X$ corresponding to $\bfp^{\bfa} \bbZ/\bfp^{\bfs}\bbZ$ is $X(\alpha_{\bfa})$, we see from assumption that  $$ q\cdot \prod_{\bm{0}\leq \bfa \leq \bfb} \kappa(X(\alpha_{\bfa}))^{m_{\bfa}}=1. $$
Since $\kappa(X_{\bm{0}})=\kappa(X)=1$, \begin{equation}\label{eq1}
q\cdot \prod_{\bm{0}<\bfa\leq \bfs} \kappa(X(\alpha_{\bfa}))^{m_{\bfa}}=1.\tag{9} 
\end{equation} 
By Lemma \ref{deg}, considering the degree of both sides of \eqref{eq1}, we obtain $$\sum_{\bm{0}<\bfa\leq \bfs} \bfp^{\bfa}\cdot \left(1-\frac{1}{\bfp^{(\bfa+\bfb-\bfs) \vee \bm{0}}}\right)\cdot m_{\bfa}=0.$$ By Lemma \ref{matrix}, we have $m_{\bfa}=0$ for any $\bm{0}<\bfa\leq \bfs$. Furthermore, since $\kappa(X)$ can take arbitrary value (for example, choosing $X$ as a cycle graph), we must have $m_{\bm{0}}=0$, a contradiction.
\end{proof}


\appendix

\section{The M\"{o}bius Functions of Finite Posets}\label{mobi fu}
		
We review the theory of the M\"{o}bius functions on finite partially ordered sets. For more details, we refer to \cite{Encomb} and \cite{Zas}.


\begin{dfn}\label{mob def}
	Let $\scL$ be a finite partially ordered set. The \tit{M\"{o}bius function} of $\scL $ is the function $$\mu_{\scL}\colon \scL \times \scL \to \bbZ$$ satisfying the following condition $$\mu_{\scL}(x,y)=\begin{cases}
	\hfill 1 \hfill & \text{if $x=y$ ,}\\
	\displaystyle{-\sum_{\substack{z\in \scL; \\ x\leq z< y}}\mu_{\scL}(x,z)} & \text{if $x<y$ ,}\\
	\hfill 0 \hfill & \text{otherwise}.
\end{cases} $$
\end{dfn}
These conditions can be equivalently written as (see Proposition 7.1.2. in \cite{Zas}): $$\mu_{\scL}(x,y)=\begin{cases}
	\hfill 1 \hfill & \text{if $x=y$ ,}\\
	\displaystyle{-\sum_{\substack{z\in \scL; \\ x< z\leq y}}\mu_{\scL}(z,y)} & \text{if $x<y$ ,}\\
	\hfill 0 \hfill  & \text{otherwise}.
\end{cases} $$


\begin{eg}
Let $n$ be a positive integer. Let $D_n$ be the set of all positive divisors of $n$. We define the order $i\leq j$ on $D_n$ by the condition that $i$ divides $j$. Then $D_n$ becomes a partially ordered set. Then the M\"{o}bius function $\mu_{D_{n}}$ of $D_n$ satisfies the following $$\mu_{D_n}(i,j)=\begin{cases}
	\hfill 1\hfill & \text{if $j/i=1$ ,}\\
	(-1)^k & \text{if $j/i$ is a product of $k$ distinct prime numbers,}\\
	\hfill 0 \hfill & \text{otherwise}.
\end{cases} $$
Thus, using the classical M\"{o}bius function $\mu$ in number theory, we may write $\mu_{D_n}(i,j) = \mu(j/i)$ if $i$ divides $j$.
\end{eg}

\begin{lem}\label{grp lat}
	The M\"{o}bius function $\mu_{\scC_{G}}$ for the partially ordered set $\scC_{G}$ of subgroups of a cyclic group $G$ satisfies
$$\mu_{\scC_{G}}(C, B)=\mu([B : C]),$$
where $\mu$ is the classical M\"{o}bius function in number theory.
\end{lem}
\begin{proof}
 Since the finite partially ordered set $\scC_{G}$ is isomorphic to the partially ordered set $D_n$, the lemma follows. 
\end{proof}
\noindent The following lemma plays a crucial role in the proof of \cref{main1}. 
\begin{prp}[M\"{o}bius Inversion Formula, {\cite[Proposition 3.7.2]{Encomb}}]\label{Mö inv}
	Let $\scL $ be a finite partially ordered set. Let $f,\ g\colon\scL\to G$ be two maps, where $G$ is an abelian group. Then we have \begin{enumerate} 
		\item \(\displaystyle{g(x)=\sum_{y\geq x}f(y)}\) if and only if \(\displaystyle{f(x)=\sum_{y\geq x}\mu_{\scL}(x,y) g(y)},\)
		\item \(\displaystyle{g(y)=\sum_{x\leq y}f(x)}\) if and only if \(\displaystyle{f(y)=\sum_{x\leq y}\mu_{\scL}(x,y)g(x)}\).
	\end{enumerate}
\end{prp}
\newpage
\begin{spacing}{1.2}
\begin{longtable}{|c|c|c|c|}
\caption{The list of groups with order at most $24$}\label{figure}\\

        \hline
        Order & Groups & not irreducibly represented & not exceptional\\ \hline\hline
        1 & $C_1$ & $\times$ & $\times$ \\ \hline
        2 & $C_2$ & $\times$ & $\times$ \\ \hline
        3 & $C_3$ & $\times$ & $\times$ \\ \hline
        4 & $C_4$ & $\times$ & $\times$ \\ \cline{2-4}
          & $C^2_2$ & $\checkmark$ & $\checkmark$ \\ \hline
        5 & $C_5$ & $\times$ & $\times$ \\ \hline
        6 & $C_6$ & $\times$ & $\times$ \\ \cline{2-4}
          & $S_3$ & $\times$ & $\checkmark$ \\ \hline
        7 & $C_7$ & $\times$ & $\times$ \\ \hline 
        8 & $C_8$ & $\times$ & $\times$ \\ \cline{2-4}
          & $D_4$ & $\times$ & $\checkmark$ \\ \cline{2-4}
          & $Q_8$ & $\times$ & $\times$ \\ \cline{2-4}
          & $C_2^3$ & $\checkmark$ & $\checkmark$ \\ \cline{2-4}
          & $C_2\times C_4 $ & $\checkmark$ & $\checkmark$ \\ \hline
        9 & $C_9$ & $\times$ & $\times$ \\ \cline{2-4}
          & $C^2_3$ & $\checkmark$ & $\checkmark$ \\ \hline
       10 & $C_{10}$ & $\times$ & $\times$ \\ \cline{2-4}
          & $D_5$ & $\times$ & $\checkmark$ \\ \hline
       11 & $C_{11}$ & $\times$ & $\times$ \\ \hline
       12 & $C_{12}$ & $\times$ & $\times$ \\ \cline{2-4}
          & $A_4$ & $\times$ & $\checkmark$ \\ \cline{2-4}
          & $D_6$ & $\times$ & $\checkmark$ \\ \cline{2-4}
          & $Dic_3$ & $\times$ & $\times$ \\ \cline{2-4}
          & $C_2\times C_6$ & $\checkmark$ & $\times$ \\ \hline
       13 & $C_{13}$ & $\times$ & $\times$ \\ \hline
       14 & $C_{14}$ & $\times$ & $\times$ \\ \cline{2-4}
          & $D_7$ & $\times$ & $\checkmark$ \\ \hline
       15 & $C_{15}$ & $\times$ & $\times$ \\ \hline
       16 & $C_{16}$ & $\times$ & $\times$ \\ \cline{2-4}
          & $D_8$ & $\times$ & $\checkmark$ \\ \cline{2-4}
          & $Q_{16}$ & $\times$ & $\times$ \\ \cline{2-4}
          & $SD_{16}$ & $\times$ & $\checkmark$ \\ \cline{2-4}
          & $M_4(2)$ & $\times$ & $\checkmark$ \\ \cline{2-4}
          & $C_4\circ D_4$ & $\times$ & $\checkmark$ \\ \cline{2-4}
          & $C_2^2\rtimes C_4$ & $\checkmark$ & $\checkmark$ \\ \cline{2-4}
          & $C_4^2$ & $\checkmark$ & $\checkmark$ \\ \cline{2-4}
          & $C_2^4$ & $\checkmark$ & $\checkmark$ \\ \cline{2-4}
          & $C_2\times C_8$ & $\checkmark$ & $\checkmark$ \\ \cline{2-4}
          & $C_2^2\times C_4$ & $\checkmark$ & $\checkmark$ \\ \cline{2-4}
          & $C_2\times D_4$ & $\checkmark$ & $\checkmark$ \\ \cline{2-4}
          & $C_2\times Q_8$ & $\checkmark$ & $\checkmark$ \\ \cline{2-4}
          & $C_4\rtimes C_4$ & $\checkmark$ & $\checkmark$ \\ \hline
       17 & $C_{17}$ & $\times$ & $\times$ \\ \hline   
       18 & $C_{18}$ & $\times$ & $\times$ \\ \cline{2-4}
          & $D_9$ & $\times$ & $\checkmark$ \\ \cline{2-4}
          & $C_3\rtimes S_3$ & $\checkmark$ & $\checkmark$ \\ \cline{2-4}
          & $C_3\times C_6$ & $\checkmark$ & $\times$ \\ \cline{2-4}
          & $C_3\times S_3$ & $\times$ & $\checkmark$ \\ \hline
       19 & $C_{19}$ & $\times$ & $\times$ \\ \hline
       20 & $C_{20}$ & $\times$ & $\times$ \\ \cline{2-4}
          & $D_{10}$ & $\times$ & $\checkmark$ \\ \cline{2-4}
          & $F_5$ & $\times$ & $\checkmark$ \\ \cline{2-4}
          & $Dic_5$ & $\times$ & $\times$ \\ \cline{2-4}
          & $C_2\times C_{10}$ & $\checkmark$ & $\times$ \\ \hline
       21 & $C_{21}$ & $\times$ & $\times$ \\ \cline{2-4}
          & $C_7\rtimes C_3$ & $\times$ & $\checkmark$ \\ \hline
       22 & $C_{22}$ & $\times$ & $\times$ \\ \cline{2-4}
          & $D_{11}$ & $\times$ & $\checkmark$ \\ \hline
       23 & $C_{23}$ & $\times$ & $\times$ \\ \hline
       24 & $C_{24}$ & $\times$ & $\times$ \\ \cline{2-4}
          & $S_4$ & $\times$ & $\checkmark$ \\ \cline{2-4}
          & $D_{12}$ & $\times$ & $\checkmark$  \\ \cline{2-4}
          & $Dic_6$ & $\times$ & $\times$ \\ \cline{2-4}
          & $SL_2(\bbF_3)$ & $\times$ & $\times$ \\ \cline{2-4}
          & $C_3\rtimes D_4$ & $\times$ & $\checkmark$  \\ \cline{2-4}
          & $C_3\rtimes C_8$ & $\times$ & $\times$ \\ \cline{2-4}
          & $C_2\times C_{12}$ & $\checkmark$ & $\times$ \\ \cline{2-4}
          & $C^2_2\times C_6$ & $\checkmark$ & $\times$ \\ \cline{2-4}
          & $C_2\times A_4$ & $\times$ & $\checkmark$ \\ \cline{2-4}
          & $C_4\times S_3$ & $\times$ & $\checkmark$  \\ \cline{2-4}
          & $C_3\times D_4$ & $\times$ & $\times$ \\ \cline{2-4}
          & $C_2^2\times S_3$ & $\checkmark$ & $\checkmark$  \\ \cline{2-4}
          & $C_3\times Q_8$ & $\times$ & $\times$ \\ \cline{2-4}
          & $C_2\times Dic_3$ & $\checkmark$ & $\times$ \\ \hline
\end{longtable}

\end{spacing}

\bibliographystyle{amsalpha} \bibliography{birational}	
\end{document}